\documentclass[12pt]{article}
\usepackage[UKenglish]{babel}
\usepackage[T1]{fontenc}
\usepackage{lmodern,amsmath,amsthm,amsfonts,amssymb,graphicx,float,microtype,thmtools,underscore,mathtools,thm-restate}
\usepackage[shortlabels]{enumitem}
\setlist[itemize]{topsep=0ex,itemsep=0ex,parsep=0ex}
\setlist[enumerate]{topsep=0ex,itemsep=0ex,parsep=0ex}
\usepackage[usenames,dvipsnames,svgnames,table]{xcolor}
\usepackage{todonotes}
\usepackage[unicode=true]{hyperref}
\usepackage{breakurl}
\hypersetup{ 
colorlinks,
linkcolor={blue!60!black},
citecolor={black},
breaklinks=true,
urlcolor={blue!60!black},
pdftitle={Induced Minors, Asymptotic Dimension, and Baker's Technique}}
\usepackage[capitalise, compress, nameinlink, noabbrev]{cleveref}
\crefname{lem}{Lemma}{Lemmas}
\crefname{thm}{Theorem}{Theorems}
\crefname{ques}{Question}{Theorems}
\crefname{cor}{Corollary}{Corollaries}
\crefname{enumi}{Item}{Items}
\crefformat{equation}{(#2#1#3)}
\Crefformat{equation}{Equation #2(#1)#3}
\crefformat{enumi}{#2#1#3}
\Crefformat{enumi}{Item (#2#1#3)}

\newcommand{\defn}[1]{\textcolor{Maroon}{\emph{#1}}}
\usepackage[longnamesfirst,numbers,sort&compress]{natbib}
\makeatletter
\def\NAT@spacechar{~}
\makeatother
\usepackage[margin=28mm]{geometry}
\renewcommand{\baselinestretch}{1.1}
\allowdisplaybreaks

\DeclarePairedDelimiter{\set}{\{}{\}} 
    
\renewcommand{\epsilon}{\varepsilon}
\renewcommand{\emptyset}{\varnothing}
\renewcommand{\geq}{\geqslant}
\renewcommand{\leq}{\leqslant}

\DeclareMathOperator{\dist}{dist}
\DeclareMathOperator{\tw}{tw}

\newcommand{\RR}{\mathbb{R}}

\newcommand{\GG}{\mathcal{G}}

\newcommand{\NN}{\mathbb{N}}

\renewcommand{\thefootnote}{\fnsymbol{footnote}}
\theoremstyle{plain}
\newtheorem{thm}{Theorem}
\newtheorem{lem}[thm]{Lemma}

\newtheorem{ques}[thm]{Question}
\newtheorem{prop}[thm]{Proposition}

\crefname{obs}{Observation}{Observations}
\newtheorem*{lem*}{Lemma}
\theoremstyle{definition}
\newtheorem{conj}[thm]{Conjecture}
\newtheorem*{conj*}{Conjecture}

\usepackage{enumitem}

\presetkeys%
{todonotes}%
{inline}{}
\date{}

\begin{document}

\title{\bf\fontsize{18pt}{18pt}\selectfont Induced Minors, Asymptotic Dimension, and Baker's Technique}

\author{Robert Hickingbotham\,\footnotemark[1]	}

\footnotetext[1]{CNRS, ENS de Lyon, Université Claude Bernard Lyon 1, LIP, UMR 5668,
Lyon, France (\texttt{rd.hickingbotham@gmail.com}).}

\maketitle
\begin{abstract}
Asymptotic dimension is a large-scale invariant of metric spaces that was introduced by Gromov (1993). We prove that every hereditary class of bounded-degree graphs that excludes some graph as a fat minor has asymptotic dimension at most $2$, which is optimal. This makes substantial progress on a question of Bonamy, Bousquet, Esperet, Groenland, Liu, Pirot, and Scott (J. Eur. Math. Soc. 2023).

The key to our proof is a notion inspired by Baker's technique (J. ACM 1994). We say that a graph class $\GG$ has bounded \defn{Baker-treewidth} if there exists a function $f \colon \NN \to \NN$ such that, for every graph $G\in \GG$, there is a layering of $G$ such that the subgraph induced by the union of any $\ell$ consecutive layers has treewidth at most $f(\ell)$. We show that every class of bounded-degree graphs that excludes some graph as an induced minor has bounded Baker-treewidth. We discuss further applications of this result to clustered colouring and the design of linear-time approximate schemes. 
\end{abstract}

\renewcommand{\thefootnote}{\arabic{footnote}}

\section{Introduction}

A graph $H$ is a \defn{minor} of a graph $G$ if it can be obtained from a subgraph of $G$ by contracting edges.\footnote{All graphs in this paper are finite and simple unless stated otherwise. See \cref{Section:Prelim} for undefined terms and notation.} If $H$ can be obtained from an induced subgraph of $G$ by contracting edges, then it is an \defn{induced minor} of $G$. We say that $G$ is \defn{$H$-(induced)-minor-free} if it excludes $H$ as an (induced) minor.

Through a monumental series of 23 papers, \citet{GraphMinors} established a rich theory of the structure of graphs defined by an excluded minor. Among their central results are the Excluded Grid Minor Theorem~\cite{robertson1986planar}, which states that graphs that exclude a planar graph as a minor have bounded treewidth, and their celebrated Graph Minor Structure Theorem~\cite{robertson2003graph}, which characterises graphs that exclude a non-planar graph as a minor.

Inspired by the success of this theory, a recent line of work aims to establish an analogous theory for graph classes defined by an excluded induced minor $H$. To date, substantial progress has been made in the case when $H$ is planar. For example, \citet{korhonen2022} established the following induced version of the Excluded Grid Minor Theorem for bounded-degree graphs:

\begin{thm}[\cite{korhonen2022}]\label{KorhonenInducedGrid}
    There is a function $f_{\tw}\colon \NN^2\to \NN$ such that, for all $k,\Delta\in \NN$, every graph with maximum degree at most $\Delta$ and treewidth at least $f_{\tw}(k,\Delta)$ contains the $(k\times k)$-grid as an induced minor.
\end{thm}

When $H$ is non-planar, however, very little is known about the structure of $H$-induced-minor-free graphs. A result of \citet{KO2004subdivision} implies that every $H$-induced-minor-free graph with no $K_{t,t}$-subgraph has bounded degeneracy. This was strengthened by \citet{Dvorak2018subdivision} who showed that such graphs have bounded expansion. More recently, \citet{KL2024InducedMinor} showed that every $H$-induced-minor-free \mbox{$m$-edge} graph has treewidth $\widetilde{O}_H(\sqrt m)$. 

A consequence of \cref{KorhonenInducedGrid} is that bounded-degree graphs that exclude the $(k\times k)$-grid as an induced minor also exclude the $(f(k)\times f(k))$-grid as a minor (for some function $f$ that depends on $\Delta$). So it is natural to suspect that a similar result might hold for non-planar graphs, where grids are replaced by subdivided complete graphs. However, \citet{KL2024InducedMinor} observed that this is false: the family of $(k \times k)$-crossing grids has maximum degree at most $8$, excludes $K_5^{(1)}$ as an induced minor, yet contains all graphs as minors (see \cref{fig:CrossingGrid}). This illustrates the challenge of formulating a structural conjecture for bounded-degree graphs that exclude a non-planar graph as an induced minor.

\begin{figure}[H]
    \centering\includegraphics[width=0.25\textwidth]{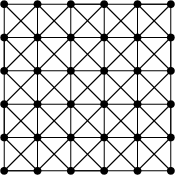}
    \caption{The $(6\times 6)$-crossing grid.}
    \label{fig:CrossingGrid}
\end{figure}

In this paper, we establish a structural theorem for bounded-degree graphs that exclude a non-planar graph as an induced minor. Our result is inspired by Baker's technique~\cite{baker1994approximation}.

\subsection{Baker's Technique}

\citet{baker1994approximation} introduced a powerful layering technique that gives linear-time approximation schemes for numerous NP-complete problems on planar graphs. Her key insight is that every planar graph has a layering such that any subgraph induced by the union of a bounded number of consecutive layers has bounded treewidth. Formally, we say that a graph class $\GG$ has bounded \defn{Baker-treewidth} if there exists a function $f \colon \NN \to \NN$ such that every graph $G\in \GG$ has a layering where any subgraph induced by the union of any $\ell$ consecutive layers has treewidth at most $f(\ell)$. We call $f$ the \defn{Baker-binding function} for $\GG$. \citet{baker1994approximation} showed that the class of planar graphs is Baker-treewidth bounded with $f(\ell)=3\ell$. 

Since many NP-complete problems admit linear-time algorithms on graphs with bounded treewidth, \citet{baker1994approximation} was able to show the following:

\begin{thm}[\cite{baker1994approximation}]\label{BakerAlgorithm}
    For any graph class with bounded Baker-treewidth, there exists a linear-time approximation scheme for maximum independent set, minimum vertex cover, minimum dominating set, and other NP-complete problems.
\end{thm}

In general, Baker's technique is applicable to any optimisation problem (maximisation or minimisation) expressible in monotone first-order logic \cite{DGKS2006approx}.

Many natural graph classes have been shown to have bounded Baker-treewidth. For example, \citet{eppstein2000diameter} characterised the minor-closed classes with bounded Baker-treewidth as those that exclude an apex graph (a graph that can be made planar by removing a single vertex). This includes any class of graphs embeddable on a fixed surface with bounded Euler genus.

Our first main contribution of this paper is to show that any class of bounded-degree $H$-induced-minor-free graphs has bounded Baker-treewidth.

\begin{restatable}{thm}{BakerTreewidthMain}\label{BakerTreewidthMain}
		For every graph $H$ and $\Delta\in \NN$, the class $\GG$ of $H$-induced-minor-free graphs with maximum degree at most $\Delta$ has bounded Baker-treewidth.
\end{restatable}

Via Baker's technique, this theorem immediately gives linear-time approximation schemes for many NP-complete problems on graphs in $\GG$. 

We note that alternative approximation schemes already exist $\GG$. Since $\GG$ admits strongly sublinear separators~\cite{KL2024InducedMinor} and has bounded maximum degree, a technique called \textit{thin system of overlaps} developed by \citet{Dvorak2018Thin} can be employed to design polynomial-time approximation for various optimisation problems on $\GG$. We highlight that Baker's technique is applicable to a broader range of problems than thin system of overlaps; see \cite{Dvorak2022meta} for a discussion contrasting these two techniques.

\begin{figure}[H]
    \centering\includegraphics[width=0.35\textwidth]{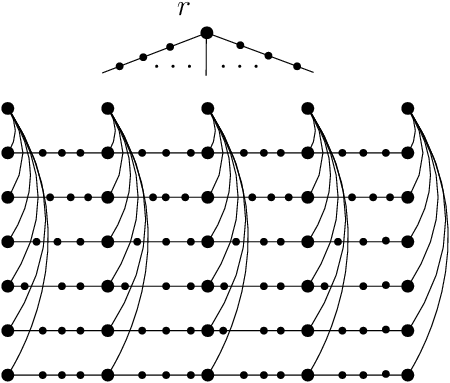}
    \caption{Construction that shows that bounded-degree in Theorem \ref{BakerTreewidthMain} is necessary. Starting with a Davies–Pohoata grid, replace each horizontal edge by a path of length $g$. Add a new vertex $r$ and join it to every vertex in the subdivided grid by a path of length $g$.}
    \label{fig:PDgrid}
\end{figure}

The bounded degree assumption in \cref{BakerTreewidthMain} is necessary, even if we assume large girth. Fix $g\in \NN$. By a simple extension of Davies--Pohoata grids \cite{Pohoata14,Davies22b} (see~\cref{fig:PDgrid}), we can construct a family of graphs $(G_{k}\colon k\in \NN)$ that excludes a tree as an induced minor where each $G_{k}$ has girth at least $g$, radius at most $g$, and treewidth at least $k$. Since any layering of a graph with bounded radius has a bounded number of non-empty layers, it follows that this family of graphs has unbounded Baker-treewidth. We omit the details.

We now discuss applications of \cref{BakerTreewidthMain} to asymptotic dimension and clustered colouring.

\subsection{Asymptotic Dimension}\label{SecIntroAsDim}

Asymptotic dimension is a large-scale invariant of metric spaces that was introduced by \citet{Gromov1993}. In the context geometric group theory, it is a quasi-isometry invariant of finitely generated groups via their Cayley graphs; see \cite{BellDranishnikov2008} for a survey. 

Let $G$ be a (possibly infinite) graph with shortest path metric $\dist_G(u,v)$. The \defn{weak-diameter} of a set $X\subseteq V(G)$ is $\max\{\dist_G(u,v)\colon u,v\in X\}$. For $\delta\geq 0$, a collection $\mathcal{U}$ of subsets of $V(G)$ is \defn{$\delta$-disjoint} if $\dist_G(X,Y)>\delta$ for all distinct $X,Y\in \mathcal{U}$. More generally, we say that it is \defn{$(k,\delta)$-disjoint} if $\mathcal{U}=\bigcup(\mathcal{U}\colon i\in [k])$ where for each $i\in [k]$, $\mathcal{U}_i$ is a $\delta$-disjoint subcollection of $\mathcal{U}$. A function $f\colon \RR^+ \to \RR^+$ is a \defn{$d$-dimensional control function} for $G$ if, for every $\delta \geq 0$, there is a $(d+1,\delta)$-disjoint partition $\mathcal{U}$ of $V(G)$ into sets with weak-diameter at most $f(\delta)$ in $G$. The \defn{asymptotic dimension} of a class of (possibly infinite) graphs $\GG$ is the minimum $d$ for which there exists a $d$-dimensional control function $f\colon \RR^+ \to \RR^+$ for all graphs $G\in \GG$, or infinite if no such $d$ exists.

There has been significant interest in determining the asymptotic dimension of minor-closed graph classes~\cite{fujiwara2021asymptotic,jorgensen2022geodesic,ostrovskii2015metric,bonamy2023asymptotic}. This line of research laid the foundations for \textit{coarse graph theory}~\cite{georgakopoulos2023graph}, the study of the large-scale geometry of graphs through the lens of Gromov's coarse geometry. 

In a landmark result, Bonamy, Bousquet, Esperet, Groenland, Liu, Pirot, and Scott~\cite{bonamy2023asymptotic} determined the asymptotic dimension of any minor-closed class with optimal bounds.

\begin{thm}[\cite{bonamy2023asymptotic}]\label{BonamyAsymptotic}
    Let $\GG$ be a class of finite or infinite graphs and let $H$ be a (finite) graph.
    \begin{enumerate}
        \item If $\GG$ excludes $H$ as a minor, then $\GG$ has asymptotic dimension at most $2$; and
        \item if $H$ is planar and $\GG$ excludes $H$ as a minor, then $\GG$ has asymptotic dimension at most $1$.
    \end{enumerate}
\end{thm}

Motivated by this result, \citet{bonamy2023asymptotic} asked whether an analogous result holds for \textit{fat minors}, a coarse variant of graph minors. 

For $r\in \NN$, an \defn{$r$-fat minor model} of a graph $H$ in a graph $G$ is a collection $(B_v \colon v \in V(H)) \cup (P_e \colon e\in E(H))$ of connected subgraphs in $G$ such that:

\begin{enumerate}[label=(F\arabic*), ref=(F\arabic*)]
    \item\label{F1} $V(B_v)\cap V(P_e)\neq \emptyset$ whenever $v$ is an endpoint of $e$ in $H$; and
    \item\label{F2} for any pair of distinct $X, Y \in \set{B_v\colon v\in V(H)}\cup \set{P_e \colon e\in E(H)}$ that are not covered by (F1), we have $\dist_G(X,Y) \geq r$.
\end{enumerate}

If $G$ contains an $r$-fat minor model of $H$, then $H$ is an \defn{$r$-fat minor} of $G$. Induced minors interpolate between fat minors and graph minors. In particular, induced minors essentially correspond to the $r=2$ case for fat minors, whereas graph minors correspond to the $r=1$ case.


\begin{ques}[\cite{bonamy2023asymptotic}]\label{QuestionFat}
    Is it true that there is a constant $d$ such that, for any $r\in \NN$ and graph $H$, the class of graphs $\GG$ with no $r$-fat $H$-minor has asymptotic dimension at most $d$?
\end{ques}

Progress on this question has been limited. It was initially believed that a technique of \citet{KPR1993minor} could be used to show that any graph class defined by an excluded fat minor has bounded asymptotic dimension \cite{bonamy2020ARXIV}. However, it has since been realised that this approach is flawed. Since then, bounded asymptotic dimension has been established for only a handful of specific graphs $H$, all of which are planar~\cite{georgakopoulos2023graph,chepoi2012constant,fujiwara2023coarse,bonamy2023asymptotic,Albrechtsen2024,NSS2025fat}. Notably, Nguyen, Scott, and Seymour~\cite{NSS2025fat} recently announced that any class of graphs excluding a tree as a fat minor has asymptotic dimension at most $1$.

Using \cref{BakerTreewidthMain}, we answer \cref{QuestionFat} in the affirmative with optimal bounds under the assumption that $\GG$ is hereditary (closed under taking induced subgraphs) and has bounded maximum degree.

\begin{restatable}{thm}{MainAsyDim}\label{MainAsyDim}
		 Let $\Delta,r\in \NN$ and $\GG$ be a hereditary class of finite or infinite graphs with maximum degree at most $\Delta$. Let $H$ be a (finite) graph.
    \begin{enumerate}
        \item If $\GG$ excludes $H$ as an $r$-fat minor, then $\GG$ has asymptotic dimension at most $2$; and
        \item if $H$ is planar and $\GG$ excludes $H$ as an $r$-fat minor, then $\GG$ has asymptotic dimension at most $1$.
    \end{enumerate}
\end{restatable}

Since $d$-dimensional grids have asymptotic dimension $d$~\cite{Gromov1993}, these bounds are optimal. Note that the planar case of \cref{MainAsyDim} follows fairly quickly from known results (\cref{KorhonenInducedGrid,BonamyAsymptotic}). Our main contribution is in resolving the non-planar case. 

Our assumption that $\GG$ is hereditary in \cref{MainAsyDim} does simplify the problem. In particular, we observe that fat minors and induced minors are essentially equivalent for hereditary classes (see \cref{FatInducedMinors}). This equivalence fails in general. For instance, the class of radius-$1$ graphs contains all graphs as induced subgraphs, but excludes $K_2$ as a $3$-fat-minor.

Finally, our bounded degree assumption is natural in the context of geometric group theory, since Cayley graphs of finitely generated group have bounded maximum degree. As such, \cref{MainAsyDim} may have applications in this setting.

\subsection{Clustered Colouring}

Our final application of \cref{BakerTreewidthMain} is to clustered colouring. The \defn{clustered chromatic number} of a graph class $\GG$ is the minimum $k\in \NN$ such that for some $c\in \NN$, every graph in $\GG$ is $k$-colourable such that each monochromatic component has at most $c$ vertices. See \cite{wood2018defective} for a survey on clustered colouring.

Using the Graph Minor Structure Theorem, \citet{LO18} proved the following:

\begin{thm}[\cite{LO18}]\label{LOclustered}
    For every graph $H$ and $\Delta\in \NN$, the class of $H$-minor-free graphs with maximum degree at most $\Delta$ has clustered chromatic number at most $3$.
\end{thm}

When $H$ is non-planar, three colours is optimal due to the famous hex lemma~\cite{Gale79}.

\citet{DEMWW22} proved the following lemma which connects layering to clustered colouring:

\begin{lem}[\cite{DEMWW22}]\label{ClusteredCol}
    There is a function $f\colon \NN^2 \to \NN$ such that, for all $\Delta,k\in \NN$, if a graph $G$ with maximum degree $\Delta$ has a layering such that every subgraph induced by the union of seven consecutive layers has treewidth at most $k$, then $G$ is $3$-colourable with clustering at most $f(\Delta,k)$.
\end{lem}

Using \cref{BakerTreewidthMain,ClusteredCol}, we deduce the following generalisation of \cref{LOclustered} to induced minors:
 
\begin{thm}\label{inducedclusterd}
    For every graph $H$ and $\Delta\in \NN$, the class of $H$-induced-minor-free graphs with maximum degree at most $\Delta$ has clustered chromatic number at most $3$.
\end{thm}

Unlike the result of \citet{LO18}, our proof of \cref{inducedclusterd} does not require the Graph Minor Structure Theorem, though it implicitly uses the Excluded Grid Minor Theorem (via \cref{KorhonenInducedGrid}) and a lemma from graph minor theory (\cref{JumpGridLemma}). Note that \cref{MainAsyDim} implies \cref{inducedclusterd}, though the proof via \cref{ClusteredCol} is conceptually simpler.

\section{Preliminaries}\label{Section:Prelim}
See the textbook by \citet{diestel2017graphtheory} for undefined terms and notations. 

\subsection{Notation}
Let $\NN=\{1,2,\dots\}$. For $a,b\in \NN$ where $a \leq b$, let $[a,b]:=\{a,a+1,\dots,b\}$ and $[b]:=[1,b]$.

Let $G$ be a graph and $X,Y\subseteq V(G)$. For $v \in V(G)$, let $N_G(v):=\{u\in V(G):uv \in E(G)\}$ and $N_G[v]:=N_G(v)\cup \{v\}$. For $X\subseteq V(G)$, we write $G[X]$ for the subgraph of $G$ induced on $X$. We refer to $X$ and $G[X]$ interchangeably, whenever there is no chance of confusion. We write $G-X$ for the graph $G[V(G)\setminus X]$. We say that $X$ and $Y$ are \defn{anti-complete} if $X\cap Y=\emptyset$ and there is no edge in $G$ with one end in $A$, the other in $B$.  An \defn{$(X, Y)$-path} is a path in $G$ with one end-vertex in $X$ and the other in $Y$. The \defn{distance} between $X$ and $Y$, denoted \defn{$\dist_G(X, Y)$} is the length of the shortest $(X, Y)$-path in $G$. 

A \defn{class} of graphs $\GG$ is a set of graphs that is closed under isomorphism. It is \defn{proper} if it is not the class of all graphs; \defn{hereditary} if it is closed under taking induced subgraphs; and \defn{minor-closed} if it is closed under taking minors. We say that a graph class has a given graphic property (e.g. bounded-degree, $H$-induced-minor-free) if every graph in the class has that property.

\subsection{Minors and Induced Minors}

A \defn{minor model} of a graph $H$ in a graph $G$ is a collection $\mathcal{W}=(W_v \colon v \in V(H))$ of pairwise vertex-disjoint connected subgraphs in $G$ such that $W_u$ and $W_u$ are adjacent whenever $uv\in E(G)$. If $W_x$ and $W_y$ are anti-complete whenever $xy\not \in E(G)$, then $\mathcal{W}$ is an \defn{induced minor model}. Each $W_u$ is called a \defn{branch set} of the model. It is folklore that $H$ is an (induced) minor of $G$ if and only if $G$ contains an (induced) minor model of $H$.

For a graph $H$, a \defn{subdivision} of $H$ is a graph $J$ obtained from $H$ by replacing each edge $uv$ by a path with end-vertices $u$ and $v$. If each edge is replaced by a path with at least one internal vertex, then $J$ is a \defn{proper subdivision}. If each path has exactly $r$ internal vertices, then $J$ is the \defn{$r$-subdivision} $H^{(r)}$ of $H$.

For a graph $G$, a \defn{tree-decomposition} of $G$ is a collection $(B_x\subseteq V(G): x \in V(T))$ indexed by a tree $T$ such that (i) for each $v \in V(G)$, $T[\{x \in V(T): v \in B_x\}]$ is a non-empty subtree of $T$; and (ii) for each $uv \in E(G)$, there is a node $x \in V(T)$ such that $u,v \in B_x$. The \defn{width} of a tree-decomposition is $\max\{|B_x|-1\colon x \in V(T)\}$. The \defn{treewidth} of a graph $G$, $\tw(G)$, is the minimum width of a tree-decomposition of $G$. Treewidth is an important parameter in algorithmic and structural graph theory and is the standard measure of how similar a graph is to a tree.

\subsection{Layering}

A \defn{layering} of a graph $G$ is an ordered partition $\mathcal{L}:=(L_0,L_1,\dots)$ of $V(G)$ such that, for every edge $vw \in E(G)$, if $v \in L_i$ and $w \in L_j$, then $|i-j|\leq 1$. $\mathcal{L}$ is a \defn{\textsc{bfs}-layering} (of $G$) if $L_0 = \{r\}$ for some \defn{root vertex} $r\in V(G)$ and $L_i=\{v\in V(G):\dist_G(v,r)=i\}$ for all $i\geq 1$. A path $P$ is \defn{vertical} (with respect to $\mathcal{L}$) if $|V(P)\cap L_i|\leq 1$ for all $i\geq 0$. 

\section{Induced Minors and Baker-Treewidth}

In this section, we establish our first main result.

\BakerTreewidthMain*

We make no attempt to optimise the Baker-binding function. Before proceeding, we present a high-level overview of its proof. 

Let $G$ be a connected bounded-degree graph that admits a \textsc{bfs}-layering with the property that a subgraph $G'$ induced by a bounded number of consecutive layers has large treewidth. Our goal is to show that $G$ contains $H$ as an induced minor. Rather than constructing an induced minor model of $H$ directly, we instead construct a model of a subdivided jump grid. A \emph{jump grid} is obtained from a grid by adding a matching between pairs of vertices within a middle row (see \cref{fig:JumpGrid}). Jump grids contain all graphs as minors (\cref{JumpMinor}), and so their proper subdivisions contain all graphs as induced minors (\cref{JumpGridLemma}). 

Applying \cref{KorhonenInducedGrid} to $G'$, we find an induced minor model of a sufficiently large grid. Take a large set of vertices coming from distinct branch sets within a central row of this grid. For each vertex $x$ in this set, we find a short vertical path that starts at $x$ and exits $N_G[V(G')]$. Using an induced version of Gallai's $\mathcal{A}$-path theorem (\cref{ApathsDelta}), we match up some of the end-vertices of these short paths with pairwise anti-complete paths that are contained in $G-N_G[V(G')]$. These \textit{jump paths} will correspond to the subdivided jump edges. 

By exploiting the bounded-degreeness of $G$, we then clean up the interface with where these jump paths attach on to the grid. In doing so, we construct an induced minor model of a large subdivided jump grid, thereby completing the proof.

The rest of this section formalises this argument.

\subsection{Jump Grids}

Let $d,k\in \NN$ with $k\geq 2d$. The \defn{$(k\times k)$-grid} is the graph with vertex-set~${[k] \times [k]}$, where distinct vertices $(v,x),(w,y)$ are adjacent if ${v=w}$ and $|x-y|=1$, or ${x=y}$ and $|v-w|=1$. Suppose $G$ is a graph that contains the $(k \times k)$-grid as a spanning subgraph. For $b\in [k]$, a \defn{row-$b$ jump} is an edge of the form $(a,b)(a',b)$ where $|a-a'|>1$. We say that $G$ is a \defn{$d$-jump grid} if there exists a row index $b\in [d,k-d]$ for which $G$ contains a set of $d$ pairwise vertex-disjoint row-$b$ jump edges (see \cref{fig:JumpGrid}). 

\begin{figure}[H]
    \centering\includegraphics[width=0.28\textwidth]{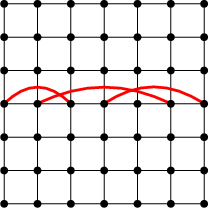}
    \caption{A $3$-jump grid.}
    \label{fig:JumpGrid}
\end{figure}

Jump grids are useful because they contain large complete graphs as minors. The following lemma is well-known within the graph minors community (for instance, see \cite[Lemma~4.3]{GSW2025GMST} for a stronger result that implies it).

\begin{lem}\label{JumpMinor}
    There exists a function $d\colon \NN\to \NN$ such that, for every $t\in \NN$, every $d(t)$-jump grid contains $K_t$ as a minor.
\end{lem}

If a graph $G$ contains a graph $H$ as a minor, then any proper subdivision $G'$ of $G$ contains $H$ as an induced minor~(see \cite{CDDFGHHWWYl2024hadwiger} for a proof). This is because the operation of edge-deletion in $G$ can be imitated by deleting subdivision vertices in $G'$. Thus, we have the following consequence of \cref{JumpMinor}.

\begin{lem}\label{JumpGridLemma}
    There exists a function $d\colon \NN\to \NN$ such that, for every $t\in \NN$, if $G$ is a $d(t)$-jump grid and $H$ is a graph with $|V(H)|\leq t$, then any proper subdivision of $G$ contains $H$ as an induced minor.
\end{lem}

\subsection{Induced $\mathcal{A}$-Paths}

Let $G$ be a graph and $\mathcal{A}\subseteq V(G)$. An \defn{$\mathcal{A}$-path} in $G$ is a path with length at least $1$ whose end-vertices are in $\mathcal{A}$. The classic $\mathcal{A}$-path theorem of \citet{Gallai1964} is the following:

\begin{thm}[\cite{Gallai1964}]\label{GallaiApath}
    For every $k\in \NN$, every graph $G$ and every $\mathcal{A}\subseteq V(G)$, $G$ contains $k$ pairwise vertex-disjoint $\mathcal{A}$-paths, or there is a set $Z\subseteq V(G)$ with $|Z|\leq 2k-2$ such that $G-Z$ has no $\mathcal{A}$-paths.
\end{thm}

When $G$ is connected and has bounded maximum degree, we have the following refinement:

\begin{prop}\label{ApathBdDegree}
     For all $\Delta,k\in \NN$, for every connected graph $G$ with maximum degree at most $\Delta$, if $\mathcal{A}\subseteq V(G)$ with $|\mathcal{A}|\geq (2k-2)(\Delta+1)+1$, then $G$ contains $k$ pairwise vertex-disjoint $\mathcal{A}$-paths.
\end{prop}

\begin{proof}
    Suppose that $G$ does not contain $k$ pairwise vertex-disjoint $\mathcal{A}$-paths. By \cref{GallaiApath}, there exists $Z\subseteq V(G)$ with $|Z|\leq 2k-2$ such that $G-Z$ has no $\mathcal{A}$-paths. Since $G$ is connected, $G-Z$ has at most $\Delta(2k-2)$ components. By the pigeonhole principle, there is a component $C$ of $G-Z$ for which $|V(C)\cap \mathcal{A}|\geq 2$, contradicting $G-Z$ having no $\mathcal{A}$-paths.
\end{proof}

In this subsection, we will prove an induced version of \cref{ApathBdDegree}. We remark that Albrechtsen, Knappe, and Wollan~\cite{WollanCoarseGraph2024} have announced an induced version of Gallai's $\mathcal{A}$-path theorem that implies our result. Since their proof have not yet appeared, we include a proof for the sake of completion. To keep it simple, we will derive it from the following induced Menger's theorem for bounded-degree graphs, independently proven by \citet{gartland2023induced} and \citet{hendrey2023induced}:

\begin{thm}[\cite{hendrey2023induced,gartland2023induced}]\label{InducedMenger}
    There exists a function $f_M\colon \NN^2 \to \NN$ such that, for every $k,\Delta \in \NN$, for every connected graph $G$ with maximum degree at most $\Delta$ and sets $X,Y\subseteq V(G)$, $G$ contains $k$ pairwise anti-complete $(X,Y)$-paths, or there exists a set $Z\subseteq V(G)$ with $|Z|\leq f_M(k,\Delta)$ such that $G-Z$ has no $(X,Y)$-path.
\end{thm}

The following is the main result of this subsection.

\begin{thm}\label{ApathsDelta}
    There exists a function $f_{\mathcal{A}}\colon \NN^2 \to \NN$ such that, for all $d,\Delta\in \NN$, if $G$ is a connected graph $G$ with maximum degree at most $\Delta$ and $\mathcal{A}\subseteq V(G)$ with $|\mathcal{A}|\geq f_{\mathcal{A}}(d,\Delta)$, then $G$ contains $d$ anti-complete $\mathcal{A}$-paths.
\end{thm}

\begin{proof}
    Set $f_{\mathcal{A}}(1,\Delta):=2$ and $f_{\mathcal{A}}(d,\Delta):=(f_{\mathcal{A}}(d-1,\Delta)+2)(\Delta+1)\cdot f_M(d,\Delta)$ for $d>1$, where $f_M$ is the function from \cref{InducedMenger}. 
    
    We proceed by induction on $d$. For $d=1$, it follows from the connectivity of $G$ that if $|\mathcal{A}|\geq 2$, then $G$ contains an $\mathcal{A}$-path. 

    Now assume $|\mathcal{A}|\geq f_{\mathcal{A}}(d,\Delta)$ and that the claim holds for $d-1$. Let $(X,Y)$ be a partition of $\mathcal{A}$ with $|X|=(\Delta+1)f_{\mathcal{A}}(d-1,\Delta)\cdot f_M(d,\Delta)$ and $|Y|=2(\Delta+1)\cdot f_M(d,\Delta)$. If $G$ contains $d$ pairwise anti-complete $(X,Y)$-paths, then we are done. Otherwise, by \cref{InducedMenger}, there exists $Z\subseteq V(G)$ with $|Z|\leq f_M(d,\Delta)$ such that $G-Z$ has no $(X,Y)$-path. 
    
    Since $G$ is connected, $G-Z$ has at most $\Delta \cdot f_M(d,\Delta)$ components. Observe that $|X\setminus Z|\geq \Delta f_{\mathcal{A}}(d-1,\Delta) f_M(d,\Delta)$ and $|Y\setminus Z|\geq 2\Delta f_M(d,\Delta)$. By the pigeonhole principle, there is a component $C_X$ of $G-Z$ with at least $f_{\mathcal{A}}(d-1,\Delta)$ vertices from $X$, and a component $C_Y$ with at least $2$ vertices from $Y$. Since $G-Z$ has no $(X,Y)$-path, $C_X$ and $C_Y$ are distinct. By induction, $C_X$ contains $d-1$ pairwise anti-complete $\mathcal{A}$-paths and $C_Y$ contains one $\mathcal{A}$-path. Since $C_X$ and $C_Y$ are anti-complete, together these give $d$ pairwise anti-complete $\mathcal{A}$-paths in $G$, as required.
\end{proof}

\subsection{Jump Grids in High Treewidth Layers}

We are now ready to prove our main lemma. Since it suffice to prove \cref{BakerTreewidthMain} for connected graphs in $\GG$, together with \cref{JumpGridLemma}, the next lemma establishes \cref{BakerTreewidthMain}.

\begin{lem}\label{MainLemma}
    There exists a function $f\colon \NN^3 \to \NN$ such that, for all $d,\Delta,\ell\in \NN$, the following holds. Let $G$ be a connected graph with maximum degree at most $\Delta$, and let $(L_0,L_1,\dots)$ be a \textsc{bfs}-layering of $G$. If there exists a subgraph $G'$ of $G$ induced by $\ell$ consecutive layers with $\tw(G')\geq f(d,\Delta,\ell)$, then $G$ contains a proper subdivision of a $d$-jump-grid as an induced minor.
\end{lem}

\begin{proof}
     Fix $d,\Delta,\ell \in \NN$. We define the following parameters:
    \begin{align*}
         k_1&:=2d\Delta(\ell+2)+3; \\
         k_2&:=f_A(d,\Delta) \, \text{(from \cref{ApathsDelta})}; \\
         k_3&:=(2\Delta^{\ell+5}(\ell+2)+1)k_2;\\
         k_4&:=2k_1\cdot k_3;\text{ and}\\
         k_5&:=f_{\tw}(k_4,\Delta) \, \text{(from \cref{KorhonenInducedGrid}).}
    \end{align*}
     
     Set $f(d,\Delta,\ell):=k_5$.

    Let $\mathcal{L}:=(L_0,L_1,\dots)$ be a \textsc{bfs}-layering of $G$ with $L_0=\{r\}$ for some root $r\in V(G)$. Suppose there exists $z \in \NN_0$ such that the subgraph $G'$ of $G$ induced by $L_z\cup L_{z+1}\cup \dots \cup L_{z+\ell-1}$ satisfies $\tw(G')\geq k_5$. If $z<2$, then $|V(G')|<\Delta^{(3+\ell)}$, contradicting $\tw(G')\geq k_5$. Hence, we may assume that $z\geq 2$. 

    By \cref{KorhonenInducedGrid}, $G'$ contains an induced minor model $\mathcal{W}_{(k_4 \times k_4)}:=(W_{(a,b)}\colon a,b\in [k_4])$ of the $(k_4 \times k_4)$-grid. 
  
    Let $\mathcal{E}$ denote the set of $k_3$ even indices in $[2k_3]$. For each $i\in \mathcal{E}$, we define the block:
    $$\mathcal{B}_i:=\{(a,b)\colon  a\in [(i-1)k_1+1, ik_1-1], b\in [(k_1-1)k_3+1, k_1k_3-1] \}.$$  
     Let ${B}_i:=\bigcup(V(W_{(a,b)})\colon (a,b)\in \mathcal{B}_i)$ and choose a vertex $v_i\in {B}_i$ for each $i\in \mathcal{E}$. 
     
     We now work towards specifying the branch sets for our $d$ subdivided jump edges. For each $i\in \mathcal{E}$, let $P_i$ be a vertical $(v_i,r)$-path in $G$. Then $|V(P_i)\cap L_j|\leq 1$ for all $j\leq z+\ell-1$. Define $P_i^{\star}:=P_i\cap( L_{z-2}\cup L_{z-1}\cup \dots \cup L_{z+\ell-1})$. Then $|V(P_i^{\star})|\leq \ell+2$.
     
     Define an auxiliary graph $J$ with $V(J)=\mathcal{E}$ where ${ij}\in E({J})$ whenever $N_G[P_i^{\star}]\cap (N_G[P_j^{\star}]\cup {B}_j)\neq \emptyset$. We claim that this graph is $(2\Delta^{\ell+5}(\ell+2))$-degenerate. If $x\in N_G[P_i^{\star}]$, then $\dist_G(x,v_i)\leq \ell+3$. So, for every vertex $x\in V(G)$, we have $|{i\in \mathcal{E}\colon x\in V(P_i^{\star})}|\leq \Delta^{\ell+4}$. Therefore, for each $i\in \mathcal{E}$, we have: 
    $$|\{j\in \mathcal{E}\colon N_G[P_i^{\star}]\cap (V(P_j^{\star})\cup {B}_j)\neq \emptyset\}|\leq
    (\Delta^{\ell+4}+1)(\Delta(\ell+2))\leq 2\Delta^{\ell+5}(\ell+2).$$
    This establishes the degeneracy bound on $J$. As such, $J$ contains an independent set $X'\subseteq \mathcal{E}$ of size $k_2$. So, for distinct $i,j\in X'$, we have:
    \begin{equation}\label{disjointBranch}
        N_G[P_i^{\star}]\cap (N_G[P_j^{\star}]\cup {B}_j)=\emptyset.
    \end{equation}
    For each $i\in X'$, define $\widetilde{P}_i:= P_i \cap( L_{0}\cup L_{1}\cup \dots \cup L_{z-2})$ and let $x_i$ be the unique vertex in $V(P_i)\cap L_{z-2}$. Define $\widetilde{G}:=G[\bigcup (V(\widetilde{P}_i)\colon i\in X')]$ and $\mathcal{A}:=\{x_i\colon i\in X'\}$. Since $N_G[P_i^{\star}]\cap N_G[P_j^{\star}]=\emptyset$ for all distinct $i,j\in X'$, each vertex $x_i$ has degree $1$. 
    
    By \cref{ApathsDelta}, $\widetilde{G}$ contains $d$ vertex-disjoint anti-complete $\mathcal{A}$-path $M_1,\dots,M_d$. These paths will correspond to the branch sets of our $d$ subdivided jump edges. Let $X\subseteq X'$ be the set of indices $i\in X$ such that $x_i$ an end-vertex of some $M_j$ path. Since each $x_i$ has degree $1$ in $\widetilde{G}$ and their neighbour is in $L_{z-3}$, it follows that, for all $i\in X$ and $j\in [d]$, $N_G[M_j]\cap V(P_i^{\star})\neq \emptyset$ if and only if $x_i$ is an end-vertex of $M_j$. Moreover, since $\widetilde{G}$ is anti-complete to $\mathcal{W}_{(k_4 \times k_4)}$, each of these paths are anti-complete to $\mathcal{W}_{(k_4 \times k_4)}$. 

    We now modify $\mathcal{W}_{(k_4 \times k_4)}$ so that we can add the subdivided jump edges. For $a\in [k_4]$, we say that column $a$ in $\mathcal{W}_{(k_4\times k_4)}$ is \defn{useless} if there exists $b\in [k_4]$ such that $V(W_{a,b})\cap \bigcup(N_G[P_i^{\star}]\colon i\in X)\neq \emptyset$. \defn{Useless} rows are defined analogously. If a column or a row is not useless, then it is \defn{useful}. Since $|N_G[P_i^{\star}]|\leq \Delta (\ell+2)$ and $|X|=2d$, at most $2d\Delta (\ell+2)$ rows are useless and at most $2d\Delta (\ell+2)$ columns are useless. Since $k_1-2> 2d\Delta (\ell+2)$, we can choose row and column index sets $\mathcal{R},\mathcal{C}\subseteq [k_4]$ satisfying:

    \begin{enumerate}[label=(\roman*), ref=(\roman*)]
        \item\label{i} $|\mathcal{R}|=|\mathcal{C}|=2 k_3$;
        \item\label{ii} for each $a\in \mathcal{R}$ and $b\in \mathcal{C}$, row $a$ is useful and column $b$ is useful; and
        \item\label{iii} for each $i\in [2 k_3]$, exactly one index of $\mathcal{R}$ (and one of $\mathcal{C})$ lies in $[(i-1)k_1+1, ik_1-1]$.
    \end{enumerate}
    
    
    Define the following sub-model of $\mathcal{W}_{(k_4\times k_4)}$:
    \begin{align*}
        \widehat{\mathcal{W}}_{(2k_3\times 2k_3)}:=&(W_{(a,b)}:a\in \mathcal{C}, \, b\in [\min\{\mathcal{R}\}, \max\{\mathcal{R}\}])\\
        &\cup (W_{(a,b)}:a\in [\min\{\mathcal{C}\},\max\{\mathcal{C}\}], \, b\in \mathcal{R}).
    \end{align*}

    \begin{figure}[b!]
        \centering\includegraphics[width=0.55\textwidth]{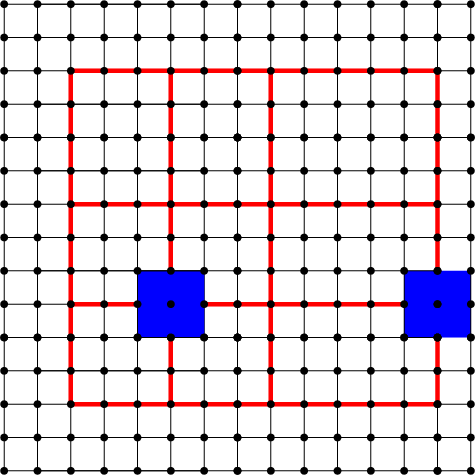}
        \caption{Constructing an induced minor model of a subdivided $(2k_3\times 2k_3)$-grid that is compatible with the jump edges: Black edges represent $\mathcal{W}_{(k_4\times k_4)}$; red edges represent $\widehat{\mathcal{W}}_{(2k_3\times 2k_3)}$; and red edges with blue subgraphs represent $\widetilde{\mathcal{W}}_{(2k_3\times 2k_3)}$. }
        \label{fig:JumpGridLayered}
    \end{figure}

    Properties \cref{i} and \cref{iii} imply that $\widehat{\mathcal{W}}_{(2k_3\times 2k_3)}$ is an induced minor model of a proper subdivision of the $(2k_3\times 2k_3)$-grid; see \cref{fig:JumpGridLayered}. Now construct $\widetilde{\mathcal{W}}_{(2k_3\times 2k_3)}$ from $\widehat{\mathcal{W}}_{(2k_3\times 2k_3)}$ by replacing, for each $i\in X$, all branch sets whose indices lie in the block $\mathcal{B}_i$ by the new branch set $G[{B}_i\cup V(P_i^{\star})\setminus\{x_i\}]$. Equation~\ref{disjointBranch} and Property \cref{ii} imply that $\widetilde{\mathcal{W}}_{(2k_3\times 2k_3)}$ is still an induced minor model of a proper subdivision of a $(2k_3\times 2k_3)$-grid, though with a change in the subdivision. Furthermore, since $X\subseteq \mathcal{E}$, the new branch sets correspond to $2d$ vertices that are in even columns of the $k_3$ row of the $(2k_3\times 2k_3)$-grid.
    
    Finally, define: 
    $$\mathcal{M}_{(2k_3\times 2k_3)}:=\widetilde{\mathcal{W}}_{(2k_3\times 2k_3)}\cup (M_j\colon j\in [d]).$$

    By the properties we have established for $\widetilde{\mathcal{W}}_{(2k_3\times 2k_3)}$ and $(M_j\colon j\in [d])$, it follows that $\mathcal{M}_{(2k_3\times 2k_3)}$ is an induced minor model of a proper subdivision of a $d$-jump grid.
\end{proof}

\section{Fat Minors and Asymptotic Dimension}

In this section, we prove our second main theorem.

\MainAsyDim*

The proof for \cref{MainAsyDim} is given in \cref{SectionProofMainAsyDim}.

\subsection{Reducing Fat Minors to Induced Minors}

We show that for hereditary graph classes, excluding a fat minor is essentially equivalent to excluding an induced minor. We refer the reader back to \cref{SecIntroAsDim} for the definition of fat minors.

\begin{lem}\label{FatInducedMinors}
    For every $r \in \NN$ and graph $H$, if a hereditary graph class $\GG$ excludes $H$ as an $r$-fat-minor, then $\GG$ excludes $H^{(3r)}$ as an induced minor.
\end{lem}

\begin{proof}
    Fix an edge $uv\in E(H)$. Let $P_{uv}$ denote the unique $(u,v)$-path of length $3r+1$ in $H^{(3r)}$. Let $\widehat{P}_{uv}$ denote the internal subpath of $P_{uv}$ obtained by deleting its two end-vertices. For each $x\in \{u,v\}$, let ${P}_{uv}^{(x)}$ be the subpath of $\widehat{P}_{uv}$ of length $r$ that is adjacent (in $H^{(3r)}$) to $x$. Then $\dist_{H^{(3r)}}({P}_{uv}^{(u)},{P}_{uv}^{(v)})= r$. Let ${P}_{uv}^{\star}$ denote the $({P}_{uv}^{(u)},{P}_{uv}^{(v)})$-subpath of $\widehat{P}_{uv}$ of length $r$. Observe that every vertex at distance at most $r$ (in $H^{(3r)}$) from ${P}_{uv}^{\star}$ lies in $V(\widehat{P}_{uv})$.
    
    Suppose, for contradiction, that some graph $G \in \GG$ contains $H^{(3r)}$ as an induced minor. Choose $G$ to be vertex-minimal with this property. Let $(B_x\colon x\in V(H^{(3r)}))$ be an induced minor model of $H^{(3r)}$ in $G$. Since $\GG$ is closed under taking induced subgraphs, minimality implies that $V(G)=\bigcup (V(B_x)\colon x\in V(H^{(3r)}))$. Consequently, for all $u,v\in V(H^{(3r)})$, we have: 
    \begin{equation}\label{DistanceEqn}
        \dist_G(B_u,B_v)\geq \dist_{H^{(3r)}}(u,v).
    \end{equation}
    
    For each $v\in V(H)$, define: 
    $$\widetilde{B}_v:=G\left[V(B_v)\cup \bigcup_{ u\in N_H(v)} \left( V(B_x)\colon x\in V({P}_{uv}^{(v)}) \right)\right].$$ \
    For each $uv\in E(H)$, define 
    $\widetilde{P}_{uv}:=G[\bigcup(B_x\colon x\in V({P}_{uv}^{\star})].$ 
    Then $(\widetilde{B}_v \colon v \in V(H)) \cup (\widetilde{P}_{uv} \colon uv\in E(H))$ is a collection of connected subgraphs in $G$ satisfying \cref{F1}. Furthermore, using \cref{DistanceEqn}, it is straightforward to check that \cref{F2} also holds. Hence $H$ is an $r$-fat minor of $G$, contradicting the assumption that $\GG$ excludes $H$ as an $r$-fat minor.
\end{proof}

A quasi-isometry is a generalisation of a bi-Lipschitz map that allows for an additive error. \citet{DHIM2024fat} showed that if a graph class excludes a graph $H$ as an $r$-fat minor, then it is quasi-isometric to a class that excludes $H$ as a $3$-fat minor. They also showed that this fails when ``$3$-fat'' is replaced by ``$2$-fat''. Since $2$-fat minors are similar to induced minors, \cref{FatInducedMinors} shows that adding a hereditary assumption significantly changes the behaviour of fat minors.

\subsection{Baker-Treewidth and Asymptotic Dimension}

We now extend a result of \citet[Theorem~5.2]{bonamy2023asymptotic} to show that any graph class with bounded Baker-treewidth has asymptotic dimension at most $2$. We emphasise that our proof is an straightforward generalisation of theirs, which is included for completion.

Let $G$ be a graph and let $L\colon V(G)\to \RR$ be a real projection. For $\ell>0$, an \defn{$\ell$-bounded component} with respect to $L$ is a maximal set $A\subseteq V(G)$ such that $G[A]$ is connected and $|L(x)-L(x')|\leq \ell$ for all $x,x'\in A$. 

Given a graph class $\GG$ and a sequence of graph classes $(\mathcal{L}_i)_{i\in \NN}$ with $\mathcal{L}_1\subseteq \mathcal{L}_2\subseteq \dots$, we say that $\GG$ is \defn{$\mathcal{L}$-layerable} if there exists a function $f\colon \RR^+ \to \NN$ such that every graph $G\in \mathcal{G}$ admits a real projection $L\colon V(G)\to \RR$ such that for every $\ell>0$, each $\ell$-bounded component in $G$ with respect to $L$ induces a graph from $\mathcal{L}_{f(\ell)}$. Note that a layering of $G$ corresponds to a real projection $L\colon V(G)\to \NN$ in which $|L(u)-L(v)|\leq 1$ for every $uv\in E(G)$.

Extending a result of \citet{BDLM2008ANdim}, \citet{bonamy2023asymptotic} showed the following.

\begin{thm}[\cite{bonamy2023asymptotic}]\label{Layerability}
    Let $n\in \NN$ and let $\mathcal{L}=(\mathcal{L}_1,\mathcal{L}_2,\dots)$ be a sequence of graph classes where each $\mathcal{L}_i$ has asymptotic dimension at most $n$. If a graph class $\GG$ is $\mathcal{L}$-layerable, then $\GG$ has asymptotic dimension at most $n+1$.
\end{thm}

We now apply \cref{Layerability} to graph classes with bounded Baker-treewidth.

\begin{thm}\label{BakerTreewidthasdim}
    Every graph class with bounded Baker-treewidth has asymptotic dimension at most $2$.
\end{thm}

\begin{proof}
    Let $\GG$ be a graph class with bounded Baker-treewidth. Then $\GG$ is $(\mathcal{L}_i)_{i\in \NN}$-layerable where $\mathcal{L}_i$ is the class of graphs of treewidth at most $i$. Since $\mathcal{L}_i$ excludes a (finite) planar graph as minor, \cref{BonamyAsymptotic} implies that $\mathcal{L}_i$ has asymptotic dimension at most $1$ for all $i\in \NN$. By applying \cref{Layerability}, it follows that $\GG$ has asymptotic dimension at most $2$.
\end{proof}

\subsection{Proof of \cref{MainAsyDim}}\label{SectionProofMainAsyDim}

To handle infinite graphs, we use the following result of \citet{bonamy2023asymptotic}. 

\begin{thm}[\cite{bonamy2023asymptotic}]\label{InftyGraph}
    For every infinite graph $G$, let $\mathcal{F}=\{G[A]\colon A\subseteq V(G), |A|<\infty\}$. Then the asymptotic dimension of $G$ is at most the asymptotic dimension of $\mathcal{F}$.
\end{thm}

We now have all the ingredients ready to prove our second main theorem.

\begin{proof}[Proof of \cref{MainAsyDim}]
    By \cref{InftyGraph} and our assumption that $\GG$ is hereditary, we may assume that all graphs in $\GG$ are finite. By \cref{FatInducedMinors}, $\GG$ excludes $H^{(3r)}$ as an induced minor. If $H$ is planar, then $\GG$ has bounded treewidth by \cref{KorhonenInducedGrid}. Thus $\GG$ excludes a planar graph as a minor, and as such, has asymptotic dimension at most $1$ by \cref{BonamyAsymptotic}. If $H$ is non-planar, then  $\GG$ has bounded Baker-treewidth by \cref{BakerTreewidthMain}. Therefore, by \cref{BakerTreewidthasdim}, it has asymptotic dimension at most $2$.
\end{proof}

\section{Conclusion and Open Problems}

We conclude with some open problems that arise from our work. See \cite{georgakopoulos2023graph} for related problems.

First, we conjecture a product structure strengthening of \cref{BakerTreewidthMain}. Graph product structure theory describes graphs as subgraphs of strong products of graphs with bounded treewidth and paths. The \defn{strong product} of graphs~$A$ and~$B$, denoted by~\defn{${A \boxtimes B}$}, is the graph with vertex-set~${V(A) \times V(B)}$, where distinct vertices ${(v,x),(w,y) \in V(A) \times V(B)}$ are adjacent if
	${v=w}$ and ${xy \in E(B)}$, or
	${x=y}$ and ${vw \in E(A)}$, or
	${vw \in E(A)}$ and~${xy \in E(B)}$. 

\citet{DJMMUW20} proved the following product structure theorem:

\begin{thm}[\cite{DJMMUW20}]\label{PGPST}
     Every planar graph is isomorphic to a subgraph of $J \boxtimes P \boxtimes K_3$ for some planar graph $J$ with $\tw(J)\leq 3$ and for some path $P$.
\end{thm}

This result has been the key to solving several long-standing open problems about queue layouts~\citep{DJMMUW20}, nonrepetitive colourings~\citep{DEJWW2020nonrepetitive}, centred colourings~\citep{debski2020improved}, and adjacency labelling schemes~\citep{DEGJMM2020adjacency}. \cref{PGPST} has been also extended for various sparse graph classes. For example, \citet{DEMWW22} proved the following:

\begin{thm}[\cite{DEMWW22}]\label{MinorBdDegree}
    For every graph $H$, there is a function $f\colon \NN \to \NN$ such that every $H$-minor-free graph $G$ with maximum degree at most $\Delta$ is isomorphic to a subgraph of $J\boxtimes P$ for some graph $J$ with $\tw(J)\leq f(\Delta)$ and for some path $P$.
\end{thm}

\cref{MinorBdDegree} implies that any class of bounded-degree $H$-minor-free graphs has bounded Baker-treewidth. In light of this, we conjecture the following common strengthening of \cref{BakerTreewidthMain,MinorBdDegree}:

\begin{conj}
    For every graph $H$, there is a function $f\colon \NN \to \NN$ such that every $H$-induced-minor-free graph $G$ with maximum degree at most $\Delta$ is isomorphic to a subgraph of $J\boxtimes P$ for some graph $J$ with $\tw(J)\leq f(\Delta)$ and for some $P$ is a path.
\end{conj}





Our final conjecture concerns the relationship between induced minors and region intersection graphs. A graph $G$ is a~\defn{region intersection graph} over a~graph $H$ if there exists a collection $\mathcal{R}=(R_v \subseteq H \colon v\in V(G))$ of connected subgraphs of $H$ such that $uv \in E(G)$ if and only if $V(R_u)\cap V(R_v)\neq \emptyset$. Region intersection graphs were introduced by Lee~\cite{Lee17} as a generalisation of the well-studied class of string graphs (intersection graph of curves on the plane). Indeed, a graph is a \defn{string graph} if and only if it is a region intersection graph over some planar graph. 

Lokshtanov~\cite{Lokshtanov-pc} and McCarty~\cite{McCarty-lectures} independently conjectured that for every graph $H$, there is a graph $H'$ such that every $H$-induced-minor-free graph is a region intersection graph over a $H'$-minor-free graph. \citet{BH2025rigs} recently disproved this conjecture by showing that for all $g,t\in \NN$, there is a~$K_{6}^{(1)}$-induced-minor-free graph of girth at least~$g$ that is not a~region intersection graph over the class of $K_t$-minor-free graphs. Since their construction has unbounded maximum degree, we revive the conjecture of Lokshtanov~\cite{Lokshtanov-pc} and McCarty~\cite{McCarty-lectures} for bounded-degree graphs.

\begin{conj}
  For every graph $H$ and $\Delta\in \NN$, there exists a graph $H'$ such that every $H$-induced-minor-free graph with maximum degree at most $\Delta$ is a region intersection graph over a $H'$-minor-free graph.
\end{conj}

In the case when $H$ is planar, this conjecture can easily be solved via \cref{KorhonenInducedGrid}. So the heart of the problem is when $H$ is non-planar.

\subsubsection*{Acknowledgement}
Thanks to {\'{E}}douard Bonnet for helpful discussions.

{
\fontsize{11pt}{12pt}
\selectfont
	
\hypersetup{linkcolor={red!70!black}}
\setlength{\parskip}{2pt plus 0.3ex minus 0.3ex}

\bibliographystyle{DavidNatbibStyle}
\bibliography{main.bbl}
}

\end{document}